\documentclass[10pt,leqno]{article}

\usepackage{verbatim}
\usepackage{amssymb, amscd}
\usepackage{mathtools}
\usepackage{rotating}
\usepackage{xcolor}
\usepackage{amsmath}
\usepackage{mathrsfs}
\usepackage{enumerate}
\usepackage[colorlinks=true, linktoc=page, citecolor=blue, linkcolor=blue, urlcolor=blue]{hyperref}

\usepackage{tabularx}
\usepackage{theorem}
\usepackage[matrix,tips,frame,color,line,poly,curve]{xy}

\newtheorem{theorem}[equation]{Theorem}
\newtheorem{corollary}[equation]{Corollary}

\newtheorem{lemma}[equation]{Lemma}

\newtheorem{proposition}[equation]{Proposition}

{\theorembodyfont{\rmfamily}
\newtheorem{remarkplain}[equation]{Remark}

}

\def\danger{\begin{trivlist}\item[]\noindent%
\begingroup\hangindent=3pc\hangafter=-2
\def\par{\endgraf\endgroup}%
\hbox to0pt{\hskip-\hangindent\dbend\hfill}\ignorespaces}
\def\enddanger{\par\end{trivlist}}

\numberwithin{equation}{section}

\newcommand{\qed}{\hfill $\square$ \medskip}
\newenvironment{proof}[1][Proof]{\noindent\textbf{#1.} }{\qed}

\newcommand{\sgn}{\mathrm{sgn}}

\renewcommand{\int}{\mathrm{int}}

\newcommand{\Ad}{\mathrm{Ad}}
\newcommand{\ad}{\mathrm{ad}}

\newcommand{\Gad}{G_\mathrm{ad}}

\newcommand{\Kad}{K_\mathrm{ad}}

\newcommand{\Cent}{\mathrm{Cent}}
\newcommand{\Stab}{\mathrm{Stab}}

\newcommand{\mH}{\mathcal H}
\renewcommand{\O}{\mathcal O}
\newcommand{\R}{\mathbb R}
\newcommand{\C}{\mathbb C}
\newcommand{\Z}{\mathbb Z}

\newcommand{\N}{\mathcal N}

\newcommand{\G}{G}

\newcommand{\n}{\mathfrak n}
\renewcommand{\sl}{\mathfrak s\mathfrak l}

\newcommand{\K}{\mathcal K}

\renewcommand{\k}{\mathfrak k}

\newcommand{\ch}[1]{#1^\vee}

\newcommand{\Lie}{\mathrm{Lie}}

\renewcommand{\t}{\mathfrak t}

\newcommand{\g}{\mathfrak g}
\newcommand{\gder}{\mathfrak g_{\mathrm{der}}}
\newcommand\inv{^{-1}}

\newcommand{\SL}{\text{SL}}

\newcommand{\s}{\mathfrak s}
\newcommand{\w}{\mathfrak w}
\newcommand{\AV}{\mathrm{AV}}
\newcommand{\KS}{\mathrm{KS}}
\newcommand{\Wh}{\mathrm{Wh}}
\newcommand{\WF}{\mathrm{WF}}
\newcommand{\AC}{\mathrm{AC}}
\newcommand{\AVann}{\mathrm{AV}_{\mathrm{ann}}}
\newcommand{\GK}{\mathrm{GK}}
\newcommand{\Op}{\O_p}
\newcommand{\Kostant}[1]{\mathcal{K}(#1)}
\newcommand{\ECom}{\mathcal{E}^{\mathcal{C}}_\Omega}
\begin{document}

\title{Nilpotent Invariants \\ for Generic Discrete Series of Real Groups}
\author{Jeffrey Adams\footnote{University of Maryland \& IDA Center for Computing Sciences ; \url{jda@math.umd.edu}} {} \& Alexandre Afgoustidis\footnote{CNRS \& Université de Lorraine, Nancy \& Metz ; \url{alexandre.afgoustidis@math.cnrs.fr}}}
\date{}

\maketitle

Let $G$ be a connected complex reductive group, defined over $\R$, and consider the group $G(\R)$ of real points. 

We are interested in representations of~$G(\R)$ which admit Whittaker models. For this we consider pairs $(B,\eta)$ where 
\begin{enumerate}[(i)]
\item $B$ is a Borel subgroup of~$G$ defined over~$\R$;
\item $\eta$ is a character of the nilpotent radical~$N(\R)$ of~$B(\R)$, subject to the condition that~$\eta$ is nontrivial on each simple root space.
\end{enumerate}
The existence of a Borel subgroup satisfying (i) is the definition of $G$ being quasisplit, which we assume holds from now on.

Recall that a representation $\pi$ of~$G(\R)$ is said to {have a $(B,\eta)$-Whittaker model} if it has a nonzero distribution vector on which $N(\R)$ acts by $\eta$ (for a more careful statement of the definition see Section~\ref{sec:invariants}).
This condition only depends on the conjugacy class of the pair, 
so we  consider pairs $(B,\eta)$ modulo conjugation by $G(\R)$, and refer to such a conjugacy class
as a {\it Whittaker datum}.
If $\w$ is a Whittaker datum, it makes sense to talk about $\pi$ having a $\w$-Whittaker model. We let $\Wh(\pi)$  be the set of
Whittaker data $\w$ such that $\pi$ has a $\w$-Whittaker model. This is a small finite set.
We say $\pi$ is {\it generic} if $\Wh(\pi)$ is non-empty.

The set $\Wh(\pi)$ is
a useful invariant of generic representations.
For instance, in the local Langlands correspondence, if $\Pi$ is a tempered $L$-packet for $G(\R)$ and $\w$ is a Whittaker datum, then $\Pi$ contains a unique representation
having a $\w$-Whittaker model. This plays a role in describing the internal structure of $L$-packets.
If $\Pi$ is a discrete series $L$-packet then each generic representation in the packet has a unique
Whittaker model, so the generic representations in~$\Pi$ are parametrized by their Whittaker models.
It is conjectured that these properties hold in the $p$-adic case as well, and this is known in many cases.

There are other invariants attached to an
irreducible representation $\pi$.  We will be concerned with two of
them, defined very differently. The first is the \emph{wavefront set} $\WF(\pi)$,
which is an analytic invariant.
The second is the \emph{associated variety} $\AV(\pi)$, which is
defined  algebraically.
Both consist of nilpotent elements of the dual of the Lie algebra of~$G$.
The invariants $\Wh(\pi), \WF(\pi)$ and~$\AV(\pi)$ are closely related to each other.

We now focus on the case of a generic discrete series representation
$\pi$ of $G(\R)$, so $\Wh(\pi)$ is a single Whittaker datum.  In this
case the invariants $\Wh(\pi), \WF(\pi)$ and $\AV(\pi)$ determine each
other. This is well known to experts, but not so easy to extract
from the literature. 

The purpose
of this paper is to give a self-contained account of these matters. More precisely, we will do two things:
\begin{enumerate}
\item Give an elementary proof that  passage from~$\pi$ to $\mathfrak{w}(\pi)$, $\WF(\pi)$ and $\AV(\pi)$ defines bijections between generic discrete series in an $L$-packet for $G(\R)$, Whittaker data for $G(\R)$, and the appropriate sets of nilpotent orbits;
\item Given one of the invariants $\Wh(\pi)$, $\WF(\pi)$ or $\AV(\pi)$, explain in the simplest possible terms how to explicitly reconstruct the other two.
\end{enumerate}
Many results in this note are well known. We give new and simplified proofs for as many of them as we could; for instance, Section~\ref{sec:Sekiguchi} contains a simple proof that $\WF(\pi)$ and $\AV(\pi)$ are exchanged by the Kostant--Sekiguchi correspondence. Otherwise we give convenient references.

\section{Some invariants of irreducible representations}

\subsection{Basic notation}

\subsubsection*{Involutions}

Let~$\sigma$ be the anti-holomorphic involution defining~$G(\R)$.
Set $\g=\Lie(G)$ and $\g(\R)=\Lie(G(\R))$. 
Choose a Cartan involution $\theta$ for $G$. This means:
$\theta$ is an algebraic involution, commuting with~$\sigma$,
and $G(\R)^\theta$ is a maximal compact subgroup of~$G(\R)$. Set $K=G^\theta$ and $K(\R)=G(\R)^\theta=K^\sigma$.
Let $\s$ be the $(-1)$-eigenspace of $\theta$ acting on~$\g$, where we use the notation $\sigma, \theta$ also for the automorphisms of~$\g$ induced by the above involutions.

\subsubsection*{Nilpotent orbits in $\g^*$} 
\label{s:dual}
Let $\N$ be the set of nilpotent elements of $\g$. Several of the invariants below are related to~$\N$, but are more naturally defined as subsets of the dual vector space~$\g^\ast$.
Therefore we need to pass to the dual Lie algebra $\g^*$, and it is important to do this as independently as possible of a choice
of isomorphism $\g\simeq\g^*$. 

First we define the nilpotent cone in $\g^*$ as follows. 
We say an element of~$\g^*$ is \emph{nilpotent} if its $G$-orbit is a (weak) cone, i.e. closed under multiplication by~$\C^*$,
and \emph{regular} if $\dim_{\Cent(G)}(X)=\mathrm{rank}(G)$.
Let $\N^*$ be the set of nilpotent elements.
The principal nilpotent orbit $\O_p\in \N^*$ is the set of regular nilpotent elements, and $\overline{\O_p}=\N^*$.

The statements of the main results below do not require the choice of an isomorphism $\g\simeq \g^*$. However such a choice appears in the proofs, so we spell this out.
Having a $G(\R)$-invariant bilinear form on $\g(\R)$ is equivalent to having a $G(\R)$-equivariant isomorphism $\g(\R)\simeq \g^*(\R)$. 
Suppose we fix a non-degenerate, $G(\R)$-invariant bilinear form $\kappa$ on $\g(\R)$. This is equivalent to having a non-degenerate bilinear form on~$\g$ which is both $G(\C)$- and Galois-invariant. The restriction of the form to each simple factor
is unique up to a nonzero constant. For example we can use the Killing form on $\gder=[\g, \g]$, extended arbitrarily by a non-degenerate form on the center.

Given such a non-degenerate bilinear form $\kappa$, we define
$\psi_\kappa: \g\rightarrow \g^*$ as usual:
$$
\psi_\kappa(X)(Y)=\kappa(X,Y)\quad (X,Y\in\g).
$$

Using $\psi_\kappa$ we shall transfer to $\g^\ast$ several notions usually defined on $\g$: the asymptotic cone, the Kostant section, and the Kostant--Sekiguchi corresondence. In each case the resulting notions are independent of the
choice of $\kappa$. See Sections~\ref{s:gds} and \ref{sec:Sekiguchi}.
In Section~\ref{s:AV} we need to be more precise: we choose  $\kappa$ to be the Killing form  (or any non-degenerate
invariant form which is negative definite on the Lie algebra of a maximal compact subgroup.)

\subsubsection*{Representations} 
We need to pass back and forth between representations of $G(\R)$ and $(\g,K)$-modules.
See \cite{greenbook} for details.

Suppose $\pi$ is an irreducible admissible representation of $G(\R)$  on a Hilbert space. 
The underlying
$(\g,K)$-module is on the space of $K$-finite vectors in
$V$; we will denote it by~$\pi_K$. 

Conversely if $\pi$ is an irreducible $(\g,K)$-module it can be
realized as the $K$-finite vectors of a Hilbert space representation, which we
refer to as the Hilbert space globalization of $\pi$.

The notion of equivalence of $(\g, K)$-modules is defined in such a way that two equivalent $(\g, K)$-modules have equivalent Hilbert space globalizations, and that any two equivalent admissible representations of~$G$ have equivalent underlying $(\g, K)$-modules. Therefore if~$\pi$ refers an equivalence class of irreducible Hilbert space representations of~$G(\R)$ (as opposed to an actual representation in the equivalence class), then it makes sense to write~$\pi_K$ for the corresponding equivalence class of $(\g, K)$-modules.

\subsection{The invariants}\label{sec:invariants}

\subsubsection*{Associated variety and Gelfand--Kirillov dimension}

Suppose $\pi$ is an irreducible $(\g,K)$-module.

The associated variety of $\pi$, denoted $\AV(\pi)$, is a closed, $K$-invariant set in~$\s^*$; so it is a finite union of
nilpotent $K$-orbits on $\s^*$. For the definition, and many properties, see \cite[Section 5]{vogan_bowdoin}.

A related invariant is the associated variety of the annihilator $I_\pi$ of $\pi$ in the universal enveloping algebra, 
denoted $\AVann(\pi)$.
This is a subset of $\N^*$ of~$\g^*$, and is the closure of a single complex nilpotent orbit: see \cite[Section 1, Theorem~4.7]{vogan_bowdoin}. Both $\AV(\pi)$ and $\AVann(\pi)$ depend only on the equivalence class of~$\pi$.

The Gelfand--Kirillov dimension~$\GK(\pi)$ of~$\pi$ can be defined in several different ways~\cite{vogan-gelfand-kirillov}; but one of them is
$$
GK(\pi)=\dim(\AV(\pi))=\frac12\dim(\AVann(\pi)).
$$
(All dimensions are complex unless otherwise stated.)

We define the associated variety of a Hilbert space representation to be that of  its underlying $(\g,K)$-module.

\subsubsection*{Wavefront set}

Now suppose $\pi$ is an irreducible admissible representation of $G(\R)$ on a Hilbert space $\mH$. 
The \emph{wave-front
  set} $\WF(\pi)$ is an analytic invariant determined by the
singularity of the global character of $\pi$ near the identity. See \cite{howe_wave_front, bv_local_structure, HarrisHeOlafsson} for the definition. 
The set~$\WF(\pi)$  is a closed $G(\R)$-invariant subset of the nilpotent cone in
$i\g(\R)^*$. As such it is the union of a finite number of nilpotent $G(\R)$-orbits. Again, this depends only on the equivalence class of~$\pi$.

We define the wave front set of a $(\g,K)$-module to be that of its 
Hilbert space globalization. 

\subsubsection*{Whittaker models}

Suppose again that $\pi$ is an irreducible admissible representation of $G(\R)$ on a Hilbert space $\mH$. We consider the corresponding actions of $G(\R)$ on the space~$\mH^\infty$ of smooth vectors in $\mH$, and on its dual $\mH^{-\infty}$. Let~$(B, \eta)$ be a pair satisfying conditions (i)-(ii) of the Introduction. We say $\pi$ has a $(B,\eta)$-Whittaker model
if there exists a nonzero element $v$ in~$\mH^{-\infty}$ such that $\pi(n)(v)=\eta(n)v$ for all $n$ in the nilradical $N(\R)$ of $B(\R)$. See~\cite{matumoto}. If~$\w$ is the corresponding Whittaker datum, i.e. the $G(\R)$-conjugacy class of the pair $(B, \eta)$, we say~$\pi$ has a $\w$-Whittaker model.  

We write $\Wh(\pi)$ for the set of Whittaker data~$\w$ such that $\pi$ has a $\w$-Whittaker model.
If $\pi$ has a unique Whittaker datum $\w$ we write $\Wh(\pi)=\w$.

We define the Whittaker data  of a $(\g,K)$-module to be that of its 
Hilbert space globalization.

\subsection{Large representations} Set $N=\dim(G)-\mathrm{rank}(G)$. This is the dimension of the nilpotent cone~$\N$. The maximal Gelfand--Kirillov dimension of a representation is $N/2$.
We say $\pi$ is {\it large} if $\GK(\pi)=N/2$. See \cite[Section~6]{vogan-gelfand-kirillov}.

\begin{lemma}
  \label{l:large}
  The following conditions are equivalent.
  \begin{enumerate}
  \item $\pi$ is large\textup{;}
        \item $\pi$ is generic\textup{;}
\item $\AVann(\pi)=\N^*=\overline{\Op}$\textup{;}
\item $\dim(\AVann(\pi))=N$\textup{;}
  \item $\dim(\AV(\pi))=N/2$\textup{;}
  \item $\GK(\pi)=N/2$.
\end{enumerate}
\end{lemma}
See \cite{vogan_bowdoin}, in particular Corollary 4.7 and Theorem 8.4.

From now on we use the terms {\it large} and {\it generic} interchangeably.

\subsection{Discrete series and Harish-Chandra parameters} \label{sec:discrete_series}

Recall $G(\R)$ has relative discrete series representations if and only if it has a relatively compact Cartan subgroup, i.e. there exists a $\theta$-stable Cartan subgroup~$T$ of~$G$  such that~$T(\R) \cap G_d$ is  compact, where $G_d$ is the derived group of~$G$. Throughout the paper, we assume $G(\R)$ does have discrete series representation, we fix  a Cartan subgroup $T \subset G$ satisfying the conditions above, and we let~$\mathfrak{t}$ denote the Lie algebra of~$T$, which is a Cartan subalgebra of~$\g$. 

Choose a  positive system~$\Delta^+$  for the roots of~$\t$ in~$\g$, and let~$\rho \in \mathfrak{t}^\ast$ be the half-sum of positive roots. Let~$\Lambda \subset \mathfrak{t}^\ast$ be the set of elements~$\lambda \in \mathfrak{t}^\ast$ which are \emph{regular} (i.e. the linear form $\lambda$ does not vanish on any coroot), and such that $\lambda-\rho$  is the differential of a character of~$T(\R)$. The action of the Weyl group~$W=W(\g, \t)$ preserves~$\Lambda$. Let~$W_K\subset W$ be the subgroup $\mathrm{Norm}_K(T)/T$.

Then any discrete series representation of~$G(\R)$ has a \emph{Harish-Chandra parameter}, which is an element of~$\Lambda/W_K$. In fact, suppose~$\lambda$ is an element of~$\Lambda$ and suppose $\zeta$ is a character of the center~$Z(\R)$ of~$G(\R)$ whose differential agrees with~$\lambda-\rho$ on the intersection $T(\R) \cap Z(\R)$; then the pair $(\lambda, \zeta)$  determines a discrete series representation~$\pi_{\lambda, \zeta}$ of~$G(\R)$. Every discrete series representation arises in this way, and $\pi_{\lambda, \zeta}=\pi_{\lambda', \zeta'}$ if and only if $\zeta=\zeta'$ and $\lambda, \lambda'$ are $W_K$-conjugate. See  \cite[Section 8]{AV1} and \cite{Contragredient} for more on the parametrization.  

Recall that the \emph{Weyl chambers} of~$\mathfrak{t}^\ast$ are the connected components of the set of regular elements. Each Weyl chamber specifies a positive root system for the roots of $\mathfrak{t}$ in~$\g$, and we say a Weyl chamber is \emph{large} if the corresponding simple roots are all noncompact. 

If~$\lambda \in \mathfrak{t}^\ast$ is a representative for the Harish-Chandra parameter of a discrete series representation $\pi$, then~$\pi$ is generic if and only if the Weyl chamber containing~$\lambda$ is large. See~\cite[Section 6]{vogan-gelfand-kirillov}.

\subsection{The invariants for generic discrete series} 

\begin{proposition} \label{invariants_ds}
  Suppose $\pi$ is a generic discrete series representation of $G(\R)$. Then\textup{:}
\begin{enumerate}[(a)]
\item $\Wh(\pi)$ is a single Whittaker datum\textup{;}
\item $\AV(\pi)$ is the closure of a single $K$-orbit on~$\mathcal{O}_p \cap \s^*$\textup{;}
\item  $\WF(\pi)$ is the closure of a single $G(\R)$-orbit on $\mathcal{O}_p \cap  i \g(\R)^*$\textup{.}
\end{enumerate}
\end{proposition}

\begin{proof} For (a), by \cite{vogan-gelfand-kirillov} and \cite{kostant_whittaker}
a representation $\pi$ is large if and only if it admits a Whittaker model for some Whittaker datum,
and by \cite[Lemma 14.14]{abv} a  large discrete series representation admits a unique Whittaker datum.

For (b) see \cite[Proposition A.9]{AV1}.

Part~(c) follows from~\cite[Lemma B]{rossmann_limit_orbits}. Let us explain how to go from there to  assertion~(c), using standard facts on the asymptotic cone recalled below (\S\,{}\ref{sec:asymp_cone}).

Let~$\mathfrak{t}^+$ be a Weyl chamber in~$\mathfrak{t}^\ast$ containing a representative~$\lambda$ for the Harish-Chandra parameter of~$\pi$. Write $\lim_{x \to 0(\mathfrak{t}^+)}\left[ G(\R)\cdot x\right]$ for the set of all elements $y \in i\g(\R)^\ast$ such that every $G(\R)$-invariant neighborhood of~$y$ contains all elements of~$\mathfrak{t}^+$ sufficiently close to zero. By Lemma B in~\cite{rossmann_limit_orbits}, if $\mathfrak{t}^+$ is a large Weyl chamber, then $\lim_{x \to 0(\mathfrak{t}^+)} \left[G(\R)\cdot x\right]$ is a single $G(\R)$-orbit.  Since this does happen when~$\pi$ is a generic discrete series representation, to prove assertion~(c)  it is enough to check the closure of $\lim_{x \to 0(\mathfrak{t}^+)} \left[G(\R)\cdot x\right]$ contains $\WF(\pi)$ in that case. Now, the equivalence between conditions  (L1) and (L2) on page 2 of~\cite{rossmann_limit_orbits} implies that the closure of $\lim_{x \to 0(\mathfrak{t}^+)} \left[G(\R)\cdot x\right]$ contains the asymptotic cone of the orbit $G(\R)\cdot \lambda$  (see  Section~\ref{sec:asymp_cone} below). Therefore assertion~(c) follows from the fact that this asymptotic cone is equal to  $\WF(\pi)$  (Lemma~\ref{lem:WF_and_AC} below).
\end{proof}

Because of Proposition~\ref{invariants_ds} we will abuse notation slightly, as follows. If~$\pi$ is a generic discrete series representation of~$G(\R)$ and~$\O$ is a $K$-orbit in~$\s^\ast$, we write $AV(\pi)=\O$ to indicate that $AV(\pi)$ is the closure of the $K$-orbit $\O$. 
Similarly, if $\O\subset i\g(\R)^*$ is a $G(\R)$-orbit, we write $\WF(\pi)=\O$ to indicate that $\WF(\pi)$ is the closure of~$\O$.

\section{The dictionary for generic discrete series}

\subsection{Statement of the dictionary}

Given a generic discrete series representation~$\pi$ we consider the invariants of~$\pi$ defined by Proposition~\ref{invariants_ds}: a Whittaker datum  $\Wh(\pi)$, a $K$-orbit $\AV(\pi)$ in $\mathcal{O}_p \cap \s^*$, and a $G(\R)$-orbit $\WF(\pi)$ in $\mathcal{O}_p \cap  i \g(\R)^*$.

\begin{theorem} \label{th:main} Suppose~$\Pi$ is an $L$-packet of discrete series for~$G(\R)$. 
The maps $\pi \mapsto \Wh(\pi)$, $\pi \mapsto \AV(\pi)$ and $\pi\mapsto \WF(\pi)$ induce bijections between\textup{:}
\begin{enumerate}
\item[(1)] The set~$\Pi_{\mathrm{gen}}$ of generic discrete series representations in~$\Pi$\textup{;}
\item[(2)] The set of Whittaker data for $G(\R)$\textup{;}
\item[(3)] The set of~$K$-orbits on $\mathcal{O}_p \cap \s^*$\textup{;}
\item[(4)] The set of~$G(\R)$-orbits on $\mathcal{O}_p \cap  i \g(\R)^*$\textup{.}
\end{enumerate}
In particular if $\pi$ is a generic discrete series representation then each of these three invariants of $\pi$ determines the other two. 
\end{theorem}

Before embarking on the proof, let us point out that the maps in the
Theorem are defined using rather deep results of representation
theory. However, we shall show in Section~\ref{sec:explicit} that all
bijections $(i) \leftrightarrow (j)$, for $i,j \in \{(1), \dots,
(4)\}$, can be spelled out in elementary terms. Furthermore we shall
prove, using elementary arguments, that (3) $\leftrightarrow$ (4)
coincides with the Kostant--Sekiguchi correspondence.

Let us now explain the strategy of the proof, which is mainly an exercise in putting together some references that are rather scattered in the literature. What we shall actually do is: introduce a finite group~$Q(G)$, and point out that 
\begin{itemize}
\item[(i)] it acts transitively on~(1), and simply transitively on (2)--(4);
\item[(ii)] the maps $\pi \mapsto \Wh(\pi)$, $\pi \mapsto \AV(\pi)$ and $\pi\mapsto \WF(\pi)$  are equivariant for these actions.
\end{itemize}
The theorem follows immediately from these two observations, and from the following basic fact on group actions.

\begin{lemma} Suppose a group $Q$ acts on two sets $A, B$ and let $f\colon A \to B$ be a $Q$-equivariant map. 
\begin{itemize}
\item[(a)] If $Q$ acts transitively on~$B$,  then~$f$ is surjective\textup{;}
\item[(b)] If $Q$ acts transitively on~$A$ and freely on~$B$, then~$f$ is injective.
\end{itemize}
In particular, if~$Q$ acts transitively on~$A$ and simply transitively on~$B$, then~$f$ must be a bijection and the action of~$Q$ on~$A$ must be simply transitive.
\end{lemma}

\subsection{The groups $Q_\sigma(G)$ and $Q_{\theta}(G)$}

\subsubsection{Definition of the two groups, and canonical isomorphism}

Let $Z=Z(G)$ be the center of $G$, and let $\Gad$ be the complex reductive group~$G/Z$.
Then $Z(\R)=Z^\sigma$ is the center of $G(\R)$. We set $\G(\R)_\ad=G(\R)/Z(\R)$.
On the other hand  $\sigma$ factors to an automorphism  of $\Gad$, still denoted~$\sigma$, and we let  $\Gad(\R)=(\Gad)^\sigma$.
The projection $G\rightarrow \Gad$ restricts to a map $G^\sigma\rightarrow (G/Z)^\sigma$, with kernel $Z^\sigma$.
This induces a canonical injection $(G^\sigma)_\ad\hookrightarrow (\Gad)^\sigma$, and we define
$$
Q_\sigma(G)=(\Gad)^\sigma/(G^\sigma)_\ad=\Gad(\R)/\G(\R)_\ad.
$$
This finite group may be viewed as a group of outer automorphisms of~$G(\R)$.

A similar discussion applies to the involution $\theta$.
Then $Z^\theta$ is the center of $K=G^\theta$ and  $\theta$ factors to $\Gad$.
We define $(G^\theta)_\ad=G^\theta/Z^\theta$ and
$$
Q_\theta(G)=(\Gad)^\theta/(G^\theta)_\ad =(\Gad)^\theta/\Kad.
$$
This is a group of outer automorphisms of $K$. 

An important fact is that $Q_\sigma(G)$ and $Q_\theta(G)$ are canonically isomorphic. This is a generalization of the well known fact that
$G(\R)/G(\R)_0\simeq K/K_0$ where the subscript $0$ denotes identity component. 

\begin{lemma}\label{lem:iso_q}
\begin{itemize}
\item[(a)] 
Every element of $Q_\sigma(G)=(\Gad)^\sigma/(G^\sigma)_\ad$ has a representative in $(\Gad)^\sigma$ which is also in $(\Gad)^\theta$.
\item[(b)] Given $x \in Q_{\sigma}(G)$, and a representative $g\in (\Gad)^\sigma \cap (\Gad)^\theta$ as in~(a), the image of~$g$ in  $Q_\theta(G)$ depends only on~$x$, and not on  the choice of~$g$. The corresponding map $x \mapsto \varphi(x)$, from $Q_{\sigma}(G)$ to $Q_{\theta}(G)$, is a group isomorphism. 
\end{itemize}
\end{lemma}

The proof uses an interpretation of $Q_\sigma(G)$ and $Q_\theta(G)$ in terms of group cohomology, which we now spell out. 

\subsubsection{Connection with group cohomology}

The groups $Q_{\sigma}(G)$ and $Q_{\theta}(G)$ are instances of the following construction: if~$\tau$ is an involutive automorphism of~$G$, then it preserves~$Z$ and induces an involution of~$\Gad$, still denoted $\tau$; we may then set $Q_{\tau}(G)=(\Gad)^{\tau}/(G^{\tau})_{\ad}$, where $(G^{\tau})_{\ad}$ is the image of~$G^{\tau}$ in~$\Gad$. In this situation~$Q_{\tau}(G)$ has a simple interpretation in Galois cohomology, which we now recall. 

Let us first set up some general notation. When~$A$ is a group and $\tau$ is an involutive automorphism of~$A$, consider the cohomology sets $H^0(\tau, A)$ and $H^1(\tau, A)$ attached to the action of $\Z/2\Z$ on~$A$ by~$\tau$. We shall view   $H^0(\tau, A)$ as the fixed-point-set $A^{\tau}$, and $H^1(\tau,A)$ as the quotient of \mbox{$A^{-\tau} = \{ a \in A, \tau(a)=a^{-1}\}$} by the equivalence relation $a \sim x a \tau(x^{-1})$ for all $x \in A$. The set $H^1(\tau, A)$ has a distinguished point $1$, coming from the identity of~$A$; but in general it is just a pointed set, and not a group if~$A$ is nonabelian.

We now take~$A=G$, and continue to assume~$\tau$ is an involutive automorphism. Then the short exact sequence $1\rightarrow Z \rightarrow G \rightarrow \Gad\rightarrow 1$ 
gives rise to a long exact sequence of pointed sets:
\begin{equation} \label{long_ptset}
1\rightarrow Z^\tau \rightarrow G^\tau \rightarrow \Gad^{\tau} \overset{\psi_\tau}{\longrightarrow} H^1(\tau,Z)\to H^1(\tau,G)\rightarrow H^1(\tau,\Gad).
\end{equation}
In~\eqref{long_ptset} the connecting map  $ \psi_\tau\colon \Gad^{\tau} \rightarrow H^1(\tau,Z)$ is defined as follows: if $\gamma \in \Gad^{\tau}$ is the coset $gZ$, then $g\tau(g^{-1}) \in Z^{-\tau}$, and $\psi_\tau$ sends~$\gamma$ to the class of $g \tau(g^{-1})$ in $H^1(\tau, Z)$. The kernel of~$\psi_\tau$ is $(G^\tau)_{\ad}$. Since~\eqref{long_ptset} is exact we deduce that $\psi_\tau$ induces a canonical  bijection
\begin{equation} \label{interp_q_cohomology} Q_{\tau}(G) \overset{\sim}{\longrightarrow}  \ker\left(H^1(\tau,Z)\overset{\varphi_\tau}{\longrightarrow} H^1(\tau,G)\right)\end{equation}
where $\varphi_\tau$ is induced by the inclusion $Z^{-\tau} \hookrightarrow G^{-\tau}$. 



\subsubsection{Proof of Lemma~\ref{lem:iso_q}}

The proof of Lemma~\ref{lem:iso_q} uses Galois cohomology for the compact real form.  Recall $\theta\sigma = \sigma\theta$, and $G^{\sigma\theta}$ is a compact real form of~$G$. 
The involutions $\sigma$ and $\theta$ coincide on $G^{\sigma\theta}$,
therefore  $H^1(\sigma, G^{\sigma\theta})$ and  $H^1(\theta, G^{\sigma\theta})$ are the same set,
which we will denote by  $H(\ast, G^{\sigma\theta})$.
Now the inclusion $G^{\sigma\theta} \hookrightarrow G$ induces canonical maps 
\begin{equation} \label{bij_cohom} H^{1}(\sigma,  G) \longleftarrow H(\ast, G^{\sigma\theta}) \longrightarrow H^{1}(\theta,  G).\end{equation}
By \cite[Corollary~4.4 and Corollary 4.7]{galois}, both of these maps are \emph{bijections}.
Replacing~$G$ by $Z$ we get bijections $H^{1}(\sigma, Z) \leftarrow H(\ast, Z^{\sigma\theta}) \rightarrow H^{1}_{\theta}(\Gamma, Z)$.
These fit with the bijections~\eqref{bij_cohom} into a commutative diagram
\[
\begin{CD}
H^{1}(\sigma, Z) @<<< H^1(\ast, Z^{\sigma\theta})@>>> H^{1}(\theta, Z)
 \\
@V{\varphi_{\sigma}}VV @VVV @V{\varphi_{\theta}}VV  \\
H^{1}(\sigma, G) @<<< H^1(\ast, G^{\sigma\theta})@>>> H^{1}(\theta,  G)
\end{CD}
\]
where $\varphi_\sigma$, $\varphi_\theta$ are the maps in \eqref{interp_q_cohomology}.

Let us now prove part~(a) of Lemma~\ref{lem:iso_q}. Let $\gamma$ be an element of $(\Gad)^\sigma$; what we need to show is that there exists $\gamma_0 \in (G^\sigma)_{\ad}$ such that $\gamma \gamma_0^{-1}$ is $\theta$-invariant. 

Let $g$ be an element of $G$ such that $\gamma = gZ$; then $z=g \sigma(g^{-1})$ is an element of $Z^{-\sigma}$. Let $\zeta$ be its class in $H^{1}(\theta, Z)$. Let $\zeta_c$ be the inverse image of $\zeta$ under the bijection $H(\ast, Z^{\sigma\theta})\to H^{1}(\sigma, Z) $, and let $z_c\in (Z^{\sigma\theta})^{-\sigma}$ be a representative of $\zeta_c$. By definition there exists $z_0 \in Z$ such that $z = z_0 z_c \sigma(z_0^{-1})$. Since $\zeta_c$ is in the kernel of the map $H^1(\ast, Z^{\sigma\theta}) \to H^1(\ast, G^{\sigma\theta})$, there must exist  $k \in G^{\sigma\theta}$  such that $z_c = k \sigma(k^{-1})$. We see that $z = g \sigma(g)^{-1}$ is equal to $z_0 k \sigma((z_0k)^{-1})$, in other words 
\[ z_0^{-1} k^{-1} g \in G^{\sigma}.\]
Let $\gamma_0$ be the image of $z_0^{-1} k^{-1} g$ in $\Gad$. Then $\gamma_0 \in (G^{\sigma})_\ad$, and $\gamma_0$ is equal to the image of $k^{-1} g$ in $\Gad$. The element $\kappa=\gamma \gamma_0^{-1}$ of $\Gad$  is therefore $\sigma$-invariant, and since $\kappa=kZ$, we also have $(\theta\sigma)(\kappa)=\kappa$; we conclude that $\theta(\kappa) = \kappa$, q.e.d. 

Let us prove~(b). To show that when $g\in (\Gad)^\sigma \cap (\Gad)^\theta$ is a representative of $x \in Q_{\sigma}(G)$,  the image of~$g$ in  $Q_\theta(G)$ depends only on~$x$, it is enough to show that  $(G^\sigma)_\ad \cap (\Gad)^\theta \subset (G^\theta)_{\ad}$. This follows from \cite[Proposition 5.4]{galois}, applied to the action of~$G$ on~$\Gad$ and the obvious involutions. We conclude that taking $x \in Q_{\sigma}(G)$ to the image $\varphi(x)$ of $g$ in $Q_{\theta}(G)$ gives a well-defined map $\varphi\colon Q_{\sigma}(G)\to Q_{\theta}(G)$, which is clearly a group homomorphism. Exchanging the roles of $\sigma$ and $\theta$, we get a well-defined map $Q_{\theta}(G) \to Q_{\sigma}(G)$, which is the inverse of $\varphi$; therefore $\varphi$ is an isomorphism.

\qed

%

\subsection{Actions of~$Q(G)$}

\subsubsection*{On representations}

The adjoint action of $\Gad^\sigma$ on~$G(\R)$ induces an action on $G(\R)$-representations, which descends to an action of  $Q_\sigma(G)$ on equivalence class of $G(\R)$-representations. If~$q$ is an element of~$Q_\sigma(G)$, we write $\pi \mapsto \pi^q$ for the action of~$q$ on equivalence classes of $G(\R)$-representations. 

Similarly the actions of $\Gad^\theta$ on~$\g$ and~$K$ induce an action of $Q_\theta(G)$ on equivalence classes of $(\g,K)$-modules, for which we use analogous notation.

These two actions are intertwined by the isomorphism of Lemma~\ref{lem:iso_q} (this follows immediately from the definitions of the actions and of that isomorphism): 

\begin{lemma}\label{lem:action_q}
Suppose $\pi$ is the equivalence class of an irreducible Hilbert space representation of $G(\R)$. 
Suppose $q\in Q_\sigma(G)$. 
Then $(\pi^q)_K\simeq (\pi_K)^{\varphi(g)}$.
\end{lemma}

Because of Lemmas~\ref{lem:iso_q} and \ref{lem:action_q}, we write $Q(G)$ for the group $Q_\sigma(G)$ or~$Q_\theta(G)$, depending on the situation.

\subsubsection*{On real nilpotent orbits} 

To define the action of~$Q_{\sigma}(G)$ on~(4), first observe that~$\Gad$ acts on~$\g$ by the adjoint action. This action preserves~$\Op$, and its restriction to~$\Gad(\R)$ preserves~$\Op\cap i\g(\R)^*$. Since elements on a given~$G(\R)_{\mathrm{ad}}$-orbit are in the same~$G(\R)$-orbit, this induces an action of $Q_{\sigma}(G)=\Gad(\R)/G(\R)_{\mathrm{ad}}$ on~$(\Op\cap i\g(\R)^*)/G(\R)$.

\begin{proposition}\label{prop:action_on_real_orbits}
\begin{enumerate} 
\item The action of~$Q_{\sigma}(G)$ on $(\Op\cap i\g(\R)^*)/G(\R)$ is simply transitive.
\item The map $\pi \mapsto \WF(\pi)$ is $Q_{\sigma}(G)$-equivariant.
\end{enumerate}
\end{proposition}

\begin{proof}
The proof of~(1) is a standard argument in group cohomology. Suppose $\omega,\omega'$ are $G(\R)$ orbits in $\Op\cap i\g(\R)^*$,
and choose $X\in \omega,X'\in\omega'$. Then there exists $g\in G$ such that $\Ad(g)(X)=X'$. Since $X\in \g(\R)^*$ the condition $X'\in i\g(\R)^*$ 
is $g\inv \sigma(g)\in \Stab_G(X)$.  Since $X$ is principal, we have $\Stab_G(X)=ZU$ where $U$ is a unipotent group. 
A standard fact is that $H^1(\sigma, U)=1$ \cite[Chap.~III, Proposition~6]{Serre_Galois}, and from this we see that, after multiplying $g$ on the right by an element of $u$, we may assume $g\inv \sigma(g)\in Z$. This is equivalent to: the image of $g$ in $\Gad$ is in $\Gad(\R)$. On the other hand $\omega=\omega'$ 
if and only if $X'=\Ad(g)X$ for some $g\in G(\R)$. This competes the proof of (1).

Part~(2) of the proposition is comes from general properties of the wavefront set in microlocal analysis, and from the fact that the action of $Q_{\sigma}(G)$ comes from automorphisms of~$G(\R)$. More precisely, if~$q$ is an element of~$Q_{\sigma}(G)$ and $\tilde{q}$ is a representative of~$q$ in $\Gad(\R)$, then the action of~$q$ on equivalence classes of $(\g, K)$-modules comes from the action of the automorphism $\mathrm{int}(\tilde{q})$ of~$G(\R)$ on $(\g, K)$-modules. Now the wavefront set of a distribution on a manifold satisfies general covariance properties under diffeomorphisms of the manifold: this follows from Hörmander's original definitions, see e.g. \cite[Section 2, p.~800]{HarrisHeOlafsson}. Applying this to the present situation, and going through the basics as in   \cite[Section~2]{HarrisHeOlafsson}, it follows that $\pi \mapsto \WF(\pi)$ is equivariant under $Q_{\sigma}(G)$.
\end{proof}


On the other hand $Q(G)$, realized as $Q_\theta(G)$, is a group of automorphisms of $(\g,K)$.
It induces actions on $\s^*$, $\Op$ and $\Op\cap\s^*$, and on $(\g,K)$-modules.

\begin{proposition}\label{prop:action_on_K_orbits}
\begin{enumerate} 
\item The action of~$Q(G)$ on $(\Op \cap \s^*)/K$ is simply transitive.
\item The map $\pi \mapsto \AV(\pi)$ is $Q(G)$-equivariant.
\end{enumerate}
\end{proposition}

The proof of~(1) is exactly the same as that  to that of Proposition~\ref{prop:action_on_real_orbits}(1), replacing~$\sigma$ by the Cartan involution $\theta$.

As for~(2), it follows directly from the definition of the associated variety. This uses filtrations of the universal enveloping algebra of~$\g$ which are all invariant under $\Gad^{\theta_{\ad}}$: see for instance the Introduction of~\cite{vogan_bowdoin}. Inspecting the definitions in \emph{loc. cit.} it becomes clear that the map $\pi \mapsto \AV(\pi)$, taking a generic discrete series $(\g,K)$-module  (or rather its equivalence class) to its associated variety, is equivariant under~$Q_{\theta}(G)$, proving~(2). 
\qed

\section{Explicit versions of the dictionary}\label{sec:explicit}

\subsection{Associated Variety of generic discrete series}
\label{s:AV}

Suppose $\pi$ is a  discrete series representation of~$G(\R)$. Let
$\lambda\in\t^*$ be a representative for the Harish-Chandra parameter of $\pi$ (Section~\ref{sec:discrete_series}). Let $\Delta$ be the set of roots of $\t$ in $\g$, 
let $\Delta^+(\lambda)=\{\alpha \in \Delta \mid \langle\lambda,\ch\alpha\rangle>0\}$,
and let $S(\lambda)\subset\Delta^+(\lambda)$ be the simple roots.
For $\alpha\in \Delta$, let $\g_\alpha$ be the corresponding root space, and choose a non-zero element $X_\alpha\in\g_\alpha$ for each $\alpha$.

In this section only we need to  choose a particular an isomorphism \mbox{$\gder\simeq \gder^*$}.
For this we let $\kappa$ be the killing form, and define $\psi_\kappa:\gder\rightarrow \gder^*$ by
$$
\psi_\kappa(X)(Y)=\kappa(X,Y)\quad (X,Y\in \gder).
$$

\begin{lemma}
\label{l:pi_to_av}
  \label{l:AV}
Suppose $\pi$ is a generic discrete series representation and $\lambda\in\t^*$ is a representative of its  Harish-Chandra parameter.
Set

\begin{equation}
  \label{e:Fpi}
  F_\pi=\sum_{\alpha\in S(\lambda)}X_{-\alpha}.
\end{equation}
Then $F_\pi\in\s$ is a regular nilpotent element, belongs to~$\gder$, and satisfies: 
$$
\AV(\pi)=K\cdot \psi_\kappa(F_\pi).
$$
\end{lemma}
See
\cite[Propositions A.7 and A.9]{AV1}.

\begin{remarkplain}
Note that $\psi_{-\kappa}(F_\pi)=-\psi_{\kappa}(F_\pi)$, and these are not necessarily $K$-conjugate.
  This is the reason we need to use (a positive multiple of) the Killing form for this formula.
\end{remarkplain}

\subsection{Wave front set of generic discrete series}\label{sec:asymp_cone}

Suppose $\pi$ is a discrete series representation of~$G(\R)$, and let $\lambda\in\t^*$ be a representative of the Harish-Chandra parameter of $\pi$.
Define
\begin{equation} \label{semisimple_orbit_HC} \mathcal{O}_{\pi}=G(\R)\cdot\lambda.\end{equation}
This depends only on~$\pi$, and not on the choice of representative~$\lambda$. 

Consider the \emph{asymptotic cone} of $\mathcal{O}_\pi$ (see e.g. \cite[Section 3]{AVAV}) :
\[ \AC(\mathcal{O}_\pi) = \left\{ v \in \g^*  :\ \text{$\exists (\varepsilon_n, x_n)\in (\R_+ \times \mathcal{O}_\pi)^\mathbb{N}$, $\varepsilon_n \to 0$ and $\varepsilon_n x_n \to v$}  \right\}. \]
This is a union of nilpotent $G(\R)$-orbits on $i\g(\R)^*$.

\begin{lemma} \label{lem:WF_and_AC}
Let~$\pi$ be a discrete series representation of $G(\R)$. The wave-front set $\WF(\pi)$ is equal to the asymptotic cone $\AC(\mathcal{O}_\pi)$.
\end{lemma}

For this see \cite{HarrisHeOlafsson}; in the special case considered here, the result is due to Rossmann and can be extracted from \cite{RossmannPicard}, Theorems B and C. \qed

In Section~\ref{sec:alternate_WF} below we will give another simple
description of~$\WF(\pi)$, based on the description of $\AV(\pi)$ in
the previous paragraph and on the Sekiguchi correspondence.

\subsection{Whittaker data of generic discrete series}
\label{s:gds}

The following discussion summarizes results from \cite[Section~3]{adams_kaletha}.

Suppose $X\in i\g(\R)^*\subset \g^*$ is a regular nilpotent element (see Section \ref{s:dual}).
Let $\overline{\mathfrak b}(\R)$ be the kernel of $X$. This is the Lie algebra of~$\overline{B}(\R)$, where
$\overline{B}$ is an $\R$-Borel subgroup of~$G$.
Choose an $\R$-Cartan subgroup $T$ of
$\overline B$, and let~$B$ be the  Borel
subgroup opposite to~$\overline{B}$, characterized by $B\cap \overline B=T$.
Then $X$ defines a unitary
character of the nilradical  $N(\R)$  of $B(\R)$ by
\[ 
\psi_X(\exp Y)=e^{X(Y)}\quad (Y\in\n(\R)).
\] 
We define $\w_X$ to be the $G(\R)$-conjugacy class of $(B(\R),\psi_X)$; as the notation indicates this is independent 
of the choice of $T$.

The correspondence $\mathfrak{w}_{X} \leftrightarrow G(\R)\cdot X$ defines a bijection between Whittaker data and
$(\Op(\R)\cap i\g(\R)^*)/G(\R)$.

We recall a recent result of the first author and Tasho Kaletha, which spells out the connection between generic discrete series and Whittaker data using  {\it Kostant sections}. Let~$X$ be a regular nilpotent element of~$\g$, and choose  $H, Y \in \g$ such that $(X,H,Y)$ is an  $\SL(2)$-triple. The Kostant section~$\Kostant{X}$ is the affine subspace $X + \mathrm{Cent}_{\g}(Y)$ of~$\g$. This depends on the choice of $\SL(2)$-triple, but if $X \in \g(\R)$, then the $G(\R)$-conjugacy class of~$\Kostant{X}$ depends only on the $G(\R)$-conjugacy class of~$X$.

We define the Kostant section in $\g^*$ using an identification $\g\simeq \g^*$ as discussed in Section \ref{s:dual}. That is we fix an equivariant isomorphism $\psi$, and define
$$
\K(X)=\psi(\K(\psi\inv(X))\quad (X\in \g^*).
$$
Since $\psi$ is unique on each simple factor up to a non-zero scalar, it is easy to see this definition is independent of the choice of $\psi$. See \cite[Section 3.2]{adams_kaletha}.

\begin{proposition}[\phantom{}{\cite[Proposition 3.3.3]{adams_kaletha}}]\label{JeffTasho_criterion}
  \label{p:Kostant}
Suppose $\pi$ is a generic discrete series representation of~$G(\R)$. Suppose~$\pi$ be a Whittaker datum for~$G(\R)$, and let~$X \in i\g(\R)^*$ be a regular nilpotent element such that $\w = \w_X$. 
Then $\pi$ is $\w$-generic if and only if $\Kostant{X}$ meets $\mathcal{O}_\pi$. 
\end{proposition}

\begin{lemma}[\phantom{}{\cite[Lemma 3.3.4]{adams_kaletha}}]
\label{lem:Kost_and_AC}
Let $\mathcal{O} \subset i\g(\R)^\ast$ be a regular semisimple orbit and  let~$X \in \g^\ast$ be a regular nilpotent element. Then $\Kostant{X}$ meets~$\mathcal{O}$ if and only if $X$ is in the asymptotic cone~$\AC(\mathcal{O})$. 
\end{lemma}

For a closely related result see \cite[Proposition 3.5]{fm}.

The two statements above, combined with Lemma \ref{lem:WF_and_AC}, give the following result,
which is a special case of the main result of \cite{matumoto} (Theorem A):
  
\begin{corollary}\label{cor:matumoto}
Suppose $\pi$ is a generic discrete series representation. Then
$$
\Wh(\pi)=\w_X\Leftrightarrow \WF(\pi)=G(\R)\cdot X.
$$
\end{corollary}

\subsection{Geometry of Kostant sections and asymptotic cones}

By Lemma \ref{l:pi_to_av} the associated variety of a discrete representation only depends on the $W_K$-orbit of any Weyl chamber containing a representative of its Harish-Chandra parameter. 
The same holds for the wave front set and Whittaker data by Theorem \ref{th:main}. This can also be proved directly 
using some geometric facts.

If~$H$ is an element of~$\g$, write $\O_H$ for the $G(\R)$-orbit of $H$.

\begin{proposition}\label{prop:AC_chamber} Let $\mathcal{C} \subset \t$ be a large open Weyl chamber.
  If $H,H'\in\mathcal{C}$  then $\AC(\mathcal{O}_H)=\AC(\mathcal{O}_{H'})$.

  Suppose $X$ is a regular nilpotent element. Then $\K(X)\cap \O_H\ne \emptyset$ if and only if $\K(X)\cap \O_{H'}\ne\emptyset$.
\end{proposition}

Before giving the proof we need a Lemma.

\begin{lemma}\label{lem:AC_containment} Let~$\Omega$ be a regular nilpotent $G(\R)$-orbit, and let $\mathcal{E}_\Omega \subset \g$ be the union of all regular semisimple orbits~$\mathcal{O}$ such that $\AC(\mathcal{O})$ contains~$\Omega$. Then $\mathcal{E}_\Omega$ is open. 
\end{lemma}

\begin{proof}

Let $F$ be an element of~$\Omega$, and let $\mathcal{O}$  be a regular semisimple orbit such that $\AC(\mathcal{O})$ contains~$\Omega$.
Then $\Kostant{F}$ meets $\mathcal{O}$ by  Lemma~\ref{lem:Kost_and_AC}. 
Fix $X$ in $\mathcal{O} \cap \Kostant{F}$.
We have $X =F+Y$ where $Y$ belongs to the centralizer $\mathrm{Cent}_\g(F)$.
Since this centralizer is transverse to the $G(\R)$-orbits and has dimension $\dim(\g(\R))-\dim(\t(\R))$,
there exists an open neighborhood of~$X$ consisting of (regular semisimple) orbits  which meet $\Kostant{F}$.
Using Lemma~\ref{lem:Kost_and_AC} again, we see that all those orbits $\mathcal{O}'$ satisfy $F \in \AC(\mathcal{O}')$,
which implies that $\Omega = G(\R) \cdot F$ is contained in $\AC(\mathcal{O}')$.
This proves that $\mathcal{E}_\Omega$ contains an open neighborhood of any of its points, hence is open.   \end{proof}

\begin{proof}[Proof of the  Proposition]
For any regular nilpotent orbit~$\Omega$ let $\ECom$  denote the set of $H \in \mathcal{C}$ such that $\AC(\mathcal{O}_H)$ contains $\Omega$; by Lemma~\ref{lem:AC_containment}  we know that   $\ECom$ is an open subset of~$\mathcal{C}$. 

Now let $\mathcal{S}$ be the unit sphere of $\t$ for the norm induced by the given invariant bilinear form on~$\g$, and let $\mathcal{S} \cap \mathcal{C}$ be the intersection of $\mathcal{S}$ with the given chamber. This is an open subset of $\mathcal{S}$. The previous arguments show that the the radial projection $\mathcal{S}_{\Omega}$ of  $E^{\mathcal{C}}_\Omega$ on $\mathcal{S}$ is an open subset of $\mathcal{S} \cap \mathcal{C}$. 

Given regular nilpotent orbits $\Omega$ and $\Omega'$, we deduce that $\mathcal{S}_\Omega \cap \mathcal{S}_{\Omega'}$ is an open subset of the $\mathcal{S} \cap \mathcal{C}$.

Now consider the set $\Sigma$ of all points on $\mathcal{S} \cap \mathcal{C}$ which arise as the radial projection of an element~$H$ corresponding to a representative in $\t^\ast$ for the Harish-Chandra parameter of some generic discrete series representation. Since the representatives for Harish-Chandra parameters of discrete series are (a translate of) an integral lattice, the set $\Sigma$ is dense in $\mathcal{S} \cap \mathcal{C}$. Therefore if the open set $\mathcal{S}_\Omega \cap \mathcal{S}_{\Omega'}$ is nonempty, there must exist at least one generic discrete series representation~$\pi$ with a Harish-Chandra parameter representative which projects to $\mathcal{S}_\Omega \cap \mathcal{S}_{\Omega'}$. By Lemma~\ref{lem:WF_and_AC} this means $\Omega$ and $\Omega'$ must be contained in $\WF(\pi)$, which implies $\Omega = \Omega'$ since the wavefront set is the closure of a single orbit if $\pi$ is a generic discrete series. 

We conclude that if $\Omega \neq \Omega'$, then $\mathcal{S}_\Omega$ and $\mathcal{S}_\Omega$ are disjoint, and therefore $\ECom$ and $\mathcal{E}^{\mathcal{C}}_{\Omega'}$ are disjoint open subsets of~$\mathcal{C}$. Since $\mathcal{C}$ is connected and equal to the union of all subsets $\mathcal{E}^{\mathcal{C}}_{\Omega}$, the chamber $\mathcal{C}$ must be equal to a single $\ECom$.

This proves the first assertion in the Proposition; the second assertion follows from the first and Lemma~\ref{lem:Kost_and_AC}.\end{proof}

\begin{corollary}
\label{c:same}
Suppose $\pi,\pi'$ are generic discrete series representations. Let  $\lambda,\lambda'$ be representatives
for the Harish-Chandra parameters of $\pi,\pi'$. If 
$\lambda,\lambda'$ are in the same  Weyl chamber then
$\pi,\pi'$ have the same associated variety, wave-front set, and Whittaker models.
\end{corollary}

\begin{proof}
This follows from Lemma \ref{l:pi_to_av} for the associated variety;
from Lemma~\ref{lem:WF_and_AC} and Proposition~\ref{prop:AC_chamber} for the wave front set, and Lemma \ref{lem:Kost_and_AC} (or Corollary \ref{cor:matumoto})
for the Whittaker model.
\end{proof}

Here is a slightly stronger statement.

\begin{corollary}
\label{l:equalWhittaker}
Suppose $\pi$, $\pi'$ are generic discrete series representations of~$G(\R)$. Let~$\lambda, \lambda'$ be representatives for the Harish-Chandra parameters of~$\pi, \pi'$ respectively, and let $\mathscr{C}, \mathscr{C}'$ be the corresponding Weyl chambers. 

Then $\WF(\pi)=\WF(\pi')$ if and only if $\mathscr{C}, \mathscr{C}'$ are conjugate by $W_K$.
The same result holds for $\Wh(\pi)$ and $\AV(\pi)$. 
\end{corollary}

\begin{proof}
If  $\mathscr{C}, \mathscr{C}'$ are in the same $W_K$-orbit, then the Harish-Chandra parameters of $\pi$, $\pi'$ have representatives in the same Weyl chamber; therefore $\Wh(\pi)=\Wh(\pi')$ by Corollary~\ref{c:same}.

Suppose $\WF(\pi)=\WF(\pi')$. By Lemma \ref{lem:Kost_and_AC} we may assume
$\lambda=w\lambda'$ for some $w\in W$, i.e. $\pi,\pi'$ are contained
in the same $L$-packet $\Pi$. In that case, by Theorem \ref{th:main} 
we know $\Wh(\pi)=\Wh(\pi')$ if and only if $\pi\simeq \pi'$.
This holds if and only if $\lambda'=w\lambda$ with $w\in W_K$, 
i.e. $\mathscr{C}$ and $\mathscr{C}'$ are $W_K$-conjugate.
\end{proof}

\subsection{$\AV$ and $\WF$: the Kostant--Sekiguchi correspondence } \label{sec:Sekiguchi}

Theorem~\ref{th:main} gives a bijection between the set $(\Op \cap \s^*)/K$ of principal nilpotent~$K$-orbits and the set $(\Op \cap i\g(\R)^*)/G(\R)$ of principal nilpotent~$G(\R)$-orbits. 

It is well known that there is a natural bijection between nilpotent $G(\R)$-orbits on $\g(\R)$ and nilpotent $K$ orbits on $\s$. 
This is the Kostant--Sekiguchi
correspondence~\cite{sekiguchi}, which we write as $\O \mapsto \KS(\O)$. It is a deep theorem of
Schmid and Vilonen (the main result of~\cite{SV1}) that the wavefront
set and the associated variety are related by this
correspondence. (As a historical aside, the Kostant--Sekiguchi
correspondence was discovered on the basis that the correspondence
$\AV(\pi) \leftrightarrow \WF(\pi)$ might be expressed in elementary
terms.)

As discussed in Section \ref{s:dual} we transfer the Kostant--Sekiguchi correspondence to $\g^*$ using  an isomorphism $\psi_\kappa$.
Explicitly, suppose $\O$ is a $G(\R)$-orbit in $\O_p\cap i\g(\R)^*$.  Then $\KS(\O)$ is defined by:
$$
\KS(\O)=\psi_\kappa(\KS(\psi_\kappa\inv(\O)).
$$
Since $\O\subset \gder$, and $\psi_\kappa$ restricted to each simple factor is unique up to non-zero scalar, it is easy to see this definition
is independent of the choice of $\kappa$. 

Here is a particular case of Schmid and Vilonen's theorem:

\begin{proposition} \label{Sekiguchi_result}
Suppose $\pi$ is a generic discrete series representation. Then
$$
\KS(\AV(\pi))=\WF(\pi).
$$
\end{proposition}

We shall do two things here. First, we will quote a result of \cite{AVAV} which makes this case of the Kostant--Sekiguchi correspondence completely explicit. Second, we shall give a new and elementary proof of Proposition~\ref{Sekiguchi_result}: this avoids the use of Schmid and Vilonen's difficult results, and instead uses the results above (in particular Proposition~\ref{JeffTasho_criterion}) as the main ingredients.

\subsubsection{The Sekiguchi correspondence for principal nilpotent orbits}\label{sec:concrete_sek}
Let us spell out the special case of the Sekiguchi correspondence which is of interest here. We follow~\cite[Section~2]{AVAV}.

Suppose $E_\R\in i\g(\R)$ is a principal nilpotent element.
We may choose $H_\R, E_\R \in \g$ such that $(H_\R,E_\R,F_\R)$ is an $\mathrm{SL}(2)$-triple and  $E_\R$  contained in $i\g(\R)$,
which implies $H_\R\in \g(\R)$. 
After conjugating by~$G(\R)$ we may further assume $\theta(E_\R)=-F_\R$. Set 
\[   E_\theta=\frac12(E_\R-F_\R+H_\R)\quad
F_\theta=\frac12(-E_\R+F_\R+H_\R).
\]
Then
$$
\KS(G(\R)\cdot E_\R)=K\cdot E_\theta.
$$
Computing the map in the other direction is similar. Suppose $E_\theta\in \s$ is a principal nilpotent element. 
Choose $F_\theta, H_\theta \in \g$ such that $(E_\theta, H_\theta, F_\theta)$ is an $SL(2)$-triple with $F_\theta \in \s$,  which implies $H_\theta\in\k$. After conjugating by $K$ we may further assume $\sigma(E_\theta)=F_\theta$.
Define 
\begin{equation}\label{def_F_R}
E_\R=\frac12(E_\theta-F_\theta-H_\theta)\quad
F_\R=\frac12(-E_\theta+F_\theta-H_\theta).
\end{equation}
Then
\begin{equation}
\label{e:KS}
\KS(K\cdot E_\theta)=G(\R)\cdot E_\R\quad\KS(K\cdot F_\theta)=G(\R)\cdot(-E_\R).
\end{equation}

We need the corresponding result in $\g^*$. As discussed in Section \ref{s:dual} this uses a choice of an isomorphism $\psi\colon \g\to \g^*$, but the result is independent of this choice.
That is for $X\in \O_p\subset \g^*$ we have
$$
\KS(K\cdot X)=\psi(\KS(K\cdot \psi\inv(X)),
$$
independent of the choice of $\psi$.

\subsubsection{A simple proof of Proposition~\ref{Sekiguchi_result}}\label{sec:sekiguchi_proof}

Let us fix a generic discrete representation~$\pi$ and let $\lambda$ be its Harish-Chandra parameter.
As in Section \ref{s:AV} we let $\kappa$ be the Killing form and $\psi\colon\gder\to \gder^*$ be the corresponding isomorphism.

Consider the principal nilpotent element~$F_\pi$ defined in~\eqref{e:Fpi}, so by Lemma~\ref{l:AV} $\AV(\pi)=K\cdot \psi(F_\pi)$. 
We use this as~$F_\theta$ in the previous subsection: define $E_\theta = \sum_{\alpha \in S(\lambda)}X_\alpha$ and $H_\theta  = [E_\theta, F_\theta]$; then $E_\theta \in \s$ and $(E_\theta, H_\theta, F_\theta)$ is an $\mathrm{SL}(2)$-triple. Define another $\mathrm{SL}(2)$-triple $(E_\R, F_\R, H_\R)$ by
\[ \begin{aligned}
  F_\R&=\frac12(-E_\theta+F_\theta-H_\theta)\\
  E_\R&=\frac12(E_\theta-F_\theta-H_\theta)\\
  H_\R&=E_\theta+F_\theta.
\end{aligned}\]
By the discussion in the previous section $\KS(K\cdot \psi(F_\theta))=\psi(\KS(K\cdot F_\theta))=\psi(G(\R)\cdot (-E_\R))=G(\R)\cdot \psi(-E_\R)$.
We need to show this equals $\WF(\pi)$. By Corollary \ref{cor:matumoto} it is enough to show $\Wh(\pi)=\w_{\psi(-E_
  \R)}$; by 
   Lemma \ref{lem:Kost_and_AC} this is equivalent to $\K(\psi(-E_\R))$ meets $G(\R)\cdot\lambda$. 

We compute $\K(\psi(-E_\R))=\psi(\K(-E_\R))\ni -E_\R-F_\R=H_\theta$ (by \eqref{def_F_R}).
Thus $\K(\psi(-E_\R))$ meets $\psi(H_\theta)$. 
Since  $\psi(H_{\theta})$ and $\lambda$ are in the same Weyl chamber, by Proposition \ref{prop:AC_chamber} we conclude 
$\K(\psi(-E_\R))$ meets $\lambda$ as required.

\subsubsection{A simple description of the wavefront set}\label{sec:alternate_WF}

As we saw the correspondence $\pi \leftrightarrow \WF(\pi)$ can be described using the asymptotic cone of the semisimple orbit of the Harish-Chandra parameter. We used this in the proof of Proposition~\ref{Sekiguchi_result}. Now, the results of Sections~\ref{s:AV}, \ref{sec:concrete_sek} and~\ref{sec:sekiguchi_proof} provide another description of $\WF(\pi)$ from the Harish-Chandra parameter: we can compose the maps (1)$\leftrightarrow$(3)$\leftrightarrow$(4). This gives the following rather concrete conclusion.
As in Section \ref{s:AV} we define $\psi:\gder\simeq \gder^*$ using the Killing form.  

\begin{lemma}\label{lem:explicit_WF}
  Let $\pi$ be a generic discrete series representation. Fix a representative  $\lambda\in \t^*$ of the Harish-Chandra parameter of $\pi$.
Define $F_\theta=F_\pi$ as in \eqref{e:Fpi}, and choose  $E_\theta$, $H_\theta$ such that $(H_\theta,E_\theta,F_\theta)$ is an $\mathrm{SL}(2)$-triple with $\sigma(E_\theta)=F_\theta$.  Then
$$
\WF(\pi)=G(\R)\cdot \psi(F_\R), \quad \text{where $F_\R = \frac i2(-E_\theta+F_\theta-H_\theta)\in\N\cap i\g(\R)$.}
$$\end{lemma}
There is a partial converse to this statement if we consider discrete series within a single $L$-packet. Fix an $L$-packet~$\Pi$ of discrete series, and consider a regular semisimple orbit~$\Omega$ in~$\Op \cap i\g(\R)^\ast$; then we can determine the generic representation~$\pi \in \Pi$ with $\WF(\pi)=\Omega$, as follows. Fix~$\xi_\R \in \Omega$, write $\xi_\R=F_\R^\ast$ with $F_\R \in \g$ and let $E_\R, H_\R \in \g$ be such that $(E_\R, H_\R, F_\R)$ is an  $\mathrm{SL}(2)$-triple with $\theta(E_\R)=-F_\R$. Set
$$
H_\theta=E_\R-F_\R.
$$
Then $H_\theta$ is a regular semisimple element, and belongs to~$\k$.
Define a Cartan subalgebra~$\t$ of~$\g$ by $\t=\Cent_{\g}(H_\theta)$, and a set of positive roots by
$$
\Delta^+=\{\alpha\in\Delta(\t,\g)\mid  \alpha(H_\theta)>0\}.
$$
Then $\pi$ is the discrete series representation in $\Pi$ whose Harish-Chandra parameter is dominant for $\Delta^+$. 

\section{Appendix: the case of $SL(2,\R)$}

We illustrate the results of Section~\ref{sec:explicit} by making the case of $SL(2,\R)$ fully explicit.

\subsection{Notation}

The  Lie algebra $\g(\R)$  of $G=SL(2,\R)$  is the space of $2\times 2$  matrices of trace~$0$.
Let $\g=\sl(2,\C)$ be the complexified Lie algebra.
Then $\g(\R)=\g(\C)^\sigma$ where $\sigma\colon X \to \overline X$ is complex conjugation. 
Set $\theta(X)=X^t$, so $\g(\C)^\theta$  is the complexified Lie algebra of a maximal compact subgroup of $SL(2,\R)$.

We begin by defining two  bases of~$\g$. Start with the following basis:
$$
E_\R=\begin{pmatrix}0&i\\0&0\end{pmatrix},\quad 
F_\R=\begin{pmatrix}0&0\\-i&0\end{pmatrix},\quad
H_\R=\begin{pmatrix}1&0\\0&-1\end{pmatrix}.
$$
These elements satisfy the standard $SL(2)$ commutation relations: 
\[ [E_\R, F_\R] = H_\R, \quad [H_\R, E_\R] = 2E_\R, \quad [H_\R, F_\R] = -2F_\R. \]

Furthermore $E_\R,F_\R$ are in the $(-1)$-eigenspace of $\sigma$, and we have
$\theta(E_\R)=F_\R$ and $\theta(H_\R)=-H_\R$.

It is helpful to keep in mind that $E_\R,F_\R$ are $G(\R)$-conjugate,
and for $c\in \R^*$ $cE_\R$ is $G(\R)$-conjugate to $E_\R$ if and only if $c>0$.

Set $g=\frac 1{\sqrt 2}\begin{pmatrix}1&i\\i&1
\end{pmatrix}$, and take $E_\theta,F_\theta,H_\theta$ be the conjugates of $E_\R,F_\R,H_\R$ by~$g$:
$$
E_\theta=\frac12\begin{pmatrix}1&i\\i&-1\end{pmatrix}, \quad
F_\theta=\frac12\begin{pmatrix}1&-i\\-i&-1\end{pmatrix},\quad
H_\theta=\begin{pmatrix}0&-i\\i&0\end{pmatrix}.
$$
These also satisfy the standard $SL(2)$ commutation relations; $E_\theta,F_\theta$ are in the $(-1)$-eigenspace of $\theta$; 
and  $\sigma(E_\theta)=F_\theta,\sigma(H_\theta)=-H_\theta$.

The change of bases are given by:
$$
\begin{aligned}
  E_\theta&=\frac12(E_\R-F_\R+H_\R)\\
  F_\theta&=\frac12(-E_\R+F_\R+H_\R)\\
  H_\theta&=-E_\R-F_\R
\end{aligned}
$$
and
$$
\begin{aligned}
  E_\R&=\frac12(E_\theta-F_\theta-H_\theta)\\
  F_\R&=\frac12(-E_\theta+F_\theta-H_\theta)\\
  H_\R&=E_\theta+F_\theta.
\end{aligned}
$$

Define the dual basis
 $\{E_\R^*,F_\R^*,H_\R^*\}$ of $\g^*$ as usual,  i.e. $E_\R^*(E_\R)=1$, $E_\R^*(F_\R)=E_\R^*(H_\R)=0$, etc. Define 
 $\{E_\theta^*,F_\theta^*,H_\theta^*\}$ similarly.
 
\medskip

We now introduce notation for tori and roots. Let $\t(\C)=\{t_z\mid z\in\C\}$ where
$$
t_z=\begin{pmatrix}0&z\\-z&0
\end{pmatrix}.
$$
Then $\t(\R)=\{t_x\mid x\in\R\}$ is the Lie algebra of a compact Cartan subalgebra of $\g$ (the corresponding Cartan subgroup of~$G(\R)$ is also a maximal compact subgroup).

For $z\in \C$ define $\lambda_z\in \t(\C)^*$ by $\lambda_z(t_x)=xz$.
Note that $t_i=-H_\theta$, $\lambda_z(t_i)=iz$ and $H_\theta^*(t_i)=-1$.
Write $\lambda_z=cH_\theta^*$ for some $c$. Evaluating both sides on $t_i=-H_\theta$ gives
$c=-iz$:
\begin{equation}
\label{e:lambdaz}
\lambda_z=-izH_\theta^*.
\end{equation}
In particular, for $k \in \Z$ we have
\begin{equation}
\label{e:lambdai}
\lambda_{ik}=kH_\theta^*.
\end{equation}

The positive root can be taken to be   $\alpha\colon t_z \to 2iz$, i.e.
$\alpha=\lambda_{2i}$; therefore $\rho=\frac12\alpha$ is~$\lambda_i$.
Since the coroot~$\ch\alpha$ satisfies $\alpha(\ch\alpha)=2$, we have
$$
\ch\alpha=t_{-i}
$$
therefore
$$
\langle \lambda_{ik},\ch\alpha\rangle=\lambda_{ik}(t_{-i})=k
$$
for $k \in \Z$. 

Viewing $\ch\alpha$ as a map from $\C$ to $\ch{\mathfrak t}$ this says
$$
\ch\alpha(z)=t_{-iz}\quad (z\in \C).
$$

Let $\pi(\lambda_{ik})$ be the discrete series representation with Harish-Chandra parameter $\lambda_{ik}$ $(k\in \Z_{\ne 0})$.

\subsection{Identifying  $\g$ and $\g^*$}

For $c\in \R^*$ define
$$
\kappa_c(X,Y)=c\mathrm{Tr}(XY)\quad (X,Y\in \g(\R)).
$$
This is a non-degenerate, $G(\R)$ invariant bilinear form on $\g(\R)$, and extends to a $\sigma$-invariant, $G(\C)$-invariant bilinear form on $\g$. 
It is well known that every such form arises this way.
The Killing form $\kappa$ is equal to $\kappa_4$, 
as is checked by computing  $\mathrm{tr}(\ad(t_z)\ad(t_w))=-8zw$, and 
 $\kappa_c(t_z,t_w)=-2czw$.

Given $\kappa_c$ we obtain an isomorphism $\g(\R)\simeq \g(\R)^*$:
$$
\psi_c(X)(Y)=\kappa_c(X,Y)\quad (X,Y\in \g(\R))
$$
and this extends to an isomorphism $\psi_c: \g\rightarrow \g^*$. 

Explicitly we compute
$$
\psi_c(H_\R)(H_\R)=2c, \psi_c(H_\R)(E_\R)= \psi_c(H_\R)(F_\R)=0
$$
so  $\psi_c(H_\R)=2cH_\R^*$.
The result is
\begin{equation}
\label{e:psic1}
\begin{aligned}
\psi_c(H_\R)&=2cH_\R^*\\
\psi_c(E_\R)&=cF_\R^*\\
\psi_c(F_\R)&=cE_\R^*.
\end{aligned}
\end{equation}
Since we obtained the ``$\theta$-basis" by conjugating, the formulas in this basis are the same:
\begin{equation}
\label{e:psic2}
\begin{aligned}
\psi_c(H_\theta)&=2cH_\theta^*\\
\psi_c(E_\theta)&=cF_\theta^*\\
\psi_c(F_\theta)&=cE_\theta^*.
\end{aligned}
\end{equation}

Recall that for $\mu\in \g^*$, we have
$$
\AC(\mu)=\psi_c^{-1}(\AC(\psi_c(\mu)))
$$
and
$$
\K(\mu)=\psi_c^{-1}(\K(\psi_c(\mu))),
$$
independent of $c$.

\subsection{Wave-front set}

Suppose~$k \in \Z_{\neq 0}$. Let us compute $\WF(\pi_{\lambda_{ik}})$ by  \cite{HarrisHeOlafsson}, which says:
$$
\WF(\pi(\lambda_{ik}))=\AC(G(\R)\cdot \lambda_{ik}).
$$

First we compute $\AC(G(\R)\cdot kH_\theta)=\AC(G(\R)\cdot\begin{pmatrix}0&-ik\\ik&0
\end{pmatrix})$ for $k\in \Z_{\ne 0}$.
Here we are taking the real asymptotic cone; since $kH_\theta\in i\g(\R)$ this is a subset of $i\g(\R)$. 
Taking the limit as $x\rightarrow \infty$ in 
$$
\begin{pmatrix}
x&0\\0&\frac 1x
\end{pmatrix}
\begin{pmatrix}0&-ik\\ik&0
\end{pmatrix}\begin{pmatrix}
\frac 1x&0\\0&x
\end{pmatrix}=
\begin{pmatrix}0&-ikx^2\\\frac{ik}{x^2}&0
\end{pmatrix}
$$
we see
$$
\AC(G(\R)\cdot kH_\theta)=-G(\R)\cdot\sgn(k)E_\R.
$$
Now we carry this computation over to $i\g(\R)^*$:
$$
\begin{aligned}
\AC(G(\R)\cdot\lambda_{ik})&=\AC(G(\R)\cdot kH_\theta^*)\\
&=\psi_c(\AC(G(\R)\cdot \psi_c\inv(kH_\theta^*))\\
&=\psi_c(-G(\R)\cdot\sgn(\frac k{2c})E_\R)\\
&=-G(\R)\cdot\sgn(\frac k{2c})cF_\R^*\\
&=-G(\R)\cdot\sgn(k)F_\R^*\\
&=G(\R)\cdot(-\sgn(k)E_\R^*).
\end{aligned}
$$
In the last line we used the fact that $E_\R^*,F_\R^*$ are $G(\R)$-conjugate.
Note that $c$ disappeared from the end result, as expected.

Therefore by 
\cite{HarrisHeOlafsson} we have
\begin{equation}
\label{e:WF}
\WF(\pi(\lambda_{ik}))=G(\R)\cdot (-\sgn(k)E_\R^*).
\end{equation}
In particular note that $\WF(\pi_{ik})=\WF(\pi_{i\ell})\Leftrightarrow \sgn(k)=\sgn(\ell)$.
We also conclude
$$
\Wh(\pi(\lambda_{ik}))=\w_X\Leftrightarrow X\sim_{G(\R)} -\sgn(k)E_\R^*.
$$

\subsection{Associated variety}

We apply Lemma \ref{l:AV}. It is easy to see
$$
F_{\pi(\lambda_{ik})}=
\begin{cases} F_\theta&\text{ if $k>0$;}\\
  E_\theta& \text{ if $k<0$.}
\end{cases}
$$
Recall that in Section \ref{s:AV} we use the isomorphism $\psi$ defined using the Killing form $\kappa_4$, which satisfies
$\psi(F_\theta)=4E^*_\theta, \psi(E_\theta)=4F^*_\theta$. 
Therefore
\begin{equation}
  \label{e:AV}
\AV(\pi(\lambda_{ik}))=
\begin{cases}K\cdot E_\theta^*&k>0;\\
  K\cdot F_\theta^*&k<0.\\
\end{cases}
\end{equation}

Let us see that this computation and that of the previous section agree with Proposition \ref{Sekiguchi_result}, which says $\KS(\AV(\pi))=\WF(\pi)$ in that case.

Suppose $\pi=\pi(\lambda_{ik})$. First assume $k>0$ and compute
$$
\begin{aligned}
\KS(\AV(\pi))&=  \KS(K\cdot E_\theta^*)\quad \text{by \eqref{e:AV}}\\
&=  \psi_c(\KS(K\cdot \psi_c\inv (E_\theta^*)))\\
  &=  \psi_c(\KS(K\cdot \frac 1c F_\theta))\\
&=\psi_c(G(\R)\cdot \frac 1c(-E_\R))\quad\text{(by \eqref{e:KS})}\\
&=  G(\R)\cdot \psi_c(\frac{-1}c F_\R))\\
&=G(\R)\cdot (-E_\R^*)\\
&=\WF(\pi)\quad \text{(by \eqref{e:WF})}
\end{aligned}
$$
which agrees with 
Proposition \ref{Sekiguchi_result}. If $k<0$, replace $E_\theta^*$ with $F_\theta^*$, and conclude $\KS(\AV(\pi))=G(\R)\cdot E_\R^*$, in accordance with
Proposition \ref{Sekiguchi_result}.

\subsection{Whittaker model}

Finally, let us check Proposition \ref{p:Kostant}, 
which says $\Wh(\pi)=\w_X\Leftrightarrow \K(X)\cap\O_\pi\ne \emptyset$. 
Assume $k>0$. Then by~\eqref{e:WF} and Corollary~\ref{cor:matumoto} we have
$$
\Wh(\pi(\lambda_{ik}))=\mathfrak{w}_{-E_\R^*}.
$$
On the other hand, we can compute
$$
\begin{aligned}
\K(-E_\R^*)&=\psi_c(\K(\psi_c\inv(-E_\R^*)))\\
&=\psi_c(\K( -\frac 1cF_\R)).\end{aligned}$$
This contains
$$ \psi_c\left(\frac 1c(-F_\R-E_\R)\right)=\psi_c(\frac 1cH_\theta)=2H_\theta^*.
$$
So we do have
$$
\K(-E_\R^*)\cap G(\R)\cdot 2H_{\theta}^*\ne\emptyset\quad (k>0).
$$
The case $k<0$ is similar; in that case we conclude
$\Wh(\pi_{\lambda_{ik}})=\mathfrak{w}_{E_\R^\ast}$, and $\K(E_\R^*)\cap G(\R)\cdot (-2H_\theta)\ne \emptyset$.

\bibliographystyle{plain}
\bibliography{Refs}

\end{document}